\pdfoutput=1
\RequirePackage{ifpdf}
\ifpdf 
\documentclass[pdftex]{sigma}
\else
\documentclass{sigma}
\fi

\newcommand{\bbD}{{\mathbb{D}}}
\newcommand{\bbN}{{\mathbb{N}}}
\newcommand{\bbR}{{\mathbb{R}}}

\newcommand{\bbZ}{{\mathbb{Z}}}
\newcommand{\bbC}{{\mathbb{C}}}

\newcommand{\bbT}{{\mathbb{T}}}
\newcommand{\bbA}{{\mathbb{A}}}

\newcommand{\cA}{{\mathcal{A}}}
\newcommand{\cB}{{\mathcal{B}}}

\newcommand{\cF}{{\mathcal{F}}}

\newcommand{\fA}{{\mathfrak{A}}}

\newcommand{\fa}{{\mathfrak{a}}}

\newcommand{\fj}{{\mathfrak{j}}}

\newcommand{\fz}{{\mathfrak{z}}}

\newcommand{\bC}{{\mathbf{C}}}

\newcommand{\ba}{{\mathbf{a}}}
\newcommand{\bb}{{\mathbf{b}}}
\newcommand{\bc}{{\mathbf{c}}}
\newcommand{\bp}{{\mathbf{p}}}

\newcommand{\z}{\zeta}

\newcommand{\vp}{{\vec{p}}}
\newcommand{\vq}{{\vec{q}}}
\newcommand{\vbp}{{\vec{\mathbf{p}}}}

\renewcommand{\Im}{\text{\rm Im}\,}
\newcommand{\GMP}{\text{\rm GMP}}
\newcommand{\A}{\text{\rm A}}
\newcommand{\tr}{\text{\rm tr}\,}

\newcommand{\Res}{\text{\rm Res}\,}

\renewcommand{\d}{{\, \rm d}}

\newcommand{\gdw}{\Leftrightarrow}

\numberwithin{equation}{section}

\newtheorem{Theorem}{Theorem}[section]
\newtheorem{Corollary}[Theorem]{Corollary}
\newtheorem{Lemma}[Theorem]{Lemma}
\newtheorem{Proposition}[Theorem]{Proposition}
 { \theoremstyle{definition}
\newtheorem{Definition}[Theorem]{Definition}

\newtheorem{Remark}[Theorem]{Remark} }

\begin{document}

\allowdisplaybreaks

\renewcommand{\thefootnote}{$\star$}

\newcommand{\arXivNumber}{1601.07303}

\renewcommand{\PaperNumber}{066}

\FirstPageHeading

\ShortArticleName{Periodic GMP Matrices}

\ArticleName{Periodic GMP Matrices\footnote{This paper is a~contribution to the Special Issue on Orthogonal Polynomials, Special Functions and Applications. The full collection is available at \href{http://www.emis.de/journals/SIGMA/OPSFA2015.html}{http://www.emis.de/journals/SIGMA/OPSFA2015.html}}}

\Author{Benjamin EICHINGER}

\AuthorNameForHeading{B.~Eichinger}

\Address{Institute for Analysis, Johannes Kepler University, Linz, Austria}
\Email{\href{mailto:benjamin.eichinger@jku.at}{benjamin.eichinger@jku.at}}

\ArticleDates{Received January 28, 2016, in f\/inal form June 29, 2016; Published online July 07, 2016}

\Abstract{We recall criteria on the spectrum of Jacobi matrices such that the corresponding isospectral torus consists of periodic operators. Motivated by those known results for Jacobi matrices, we def\/ine a new class of operators called GMP matrices. They form a certain \textbf{G}eneralization of matrices related to the strong \textbf{M}oment \textbf{P}roblem. This class allows us to give a parametrization of \emph{almost periodic} f\/inite gap Jacobi matrices by \emph{periodic} GMP matrices. Moreover, due to their structural similarity we can carry over numerous results from the direct and inverse spectral theory of periodic Jacobi matrices to the class of periodic GMP matrices. In particular, we prove an analogue of the remarkable ``magic formula'' for this new class.}

\Keywords{spectral theory; periodic Jacobi matrices; bases of rational functions; functional models}

\Classification{30E05; 30F15; 47B36; 42C05; 58J53}

\renewcommand{\thefootnote}{\arabic{footnote}}
\setcounter{footnote}{0}

\section{Introduction}
We start by recalling some known facts from the spectral theory of Jacobi matrices; see \cite[Chapter~5]{SimonSzego}. Let $\d\sigma_+$ be a real scalar compactly supported measure and $\{P_n(x)\}_{n\geq 0}$ the corresponding orthonormal polynomials, which we obtain by orthonormalizing the monomials
\begin{gather*}
	1,\,x,\,x^2,\,\dots.
\end{gather*}
It is easy to see that they obey
\begin{gather*}
	xP_n(x)=a_nP_{n-1}(x)+b_nP_n(x)+a_{n+1}P_{n+1}(x),\qquad a_n>0,
\end{gather*}
that is, the multiplication by the independent variable in the basis $\{P_n(x)\}_{n\geq 0}$ has the matrix
\begin{gather*}
	J_+=\left[
	\begin{matrix}
		 b_0 & a_1 &0& \\
		a_1 & b_1& a_2 & \\
		0 & \ddots & \ddots & \ddots \\
		& &\ddots&\ddots
	\end{matrix}
	\right],
\end{gather*}
where $|a_n|,|b_n|\leq C$ for $C$ such that $\d\sigma_+$ has support $[-C,C]$.
Matrices of this sort are called one-sided Jacobi matrices. In general, we call an operator one-sided if it is an operator on $\ell_+^2=\ell^2(\bbZ_{\geq 0})$ and correspondingly two-sided if it acts on $\ell^2=\ell^2(\bbZ)$. Moreover, let $\{e_n\}_{n\in\bbZ}$ denote the standard basis of $\ell^2$ and $\ell_-^2=\ell^2\ominus\ell^2_+$ with the classical embedding of~$\ell_+^2$ into $\ell^2$. By
\begin{gather*}
	\big\langle(J_+-z)^{-1}e_0,e_0\big\rangle=\int\frac{\d\sigma_+(x)}{x-z},
\end{gather*}
one can associate to every one-sided Jacobi matrix a measure $\d\sigma_+$ and in fact, this describes a one-to-one correspondence between real scalar compactly supported measures and one-sided (bounded) Jacobi matrices. To deal with periodic Jacobi matrices (i.e., Jacobi matrices with periodic coef\/f\/icient sequences) it appears to be useful to extend them naturally to two-sided Jacobi matrices and in the following we will derive another basis, which turned out to be more suitable in their spectral theory than polynomials. This technique applied to ref\/lectionless Jacobi matrices with homogeneous spectra was suggested by Sodin and Yuditskii~\cite{SoYud94}. Let~$J_+$ be a periodic Jacobi matrix, $J$ its two-sided extension and $J_-=P_- JP_-^*$, where~$P_-$ denotes the orthogonal projection onto~$\ell_-^2$. One can show that there exists a polynomial,~$T_p$, of degree~$p$ such that the spectrum, $E$, of $J$ is given by
 \begin{gather}\label{eq:PolynomialJacobi}
	1)\ E=T_p^{-1}([-2,2]),
\end{gather}
2) all critical points of $T_p$ (i.e., zeros of $T_p'$) are real and 3) $|T_p(c)|\geq2$ for all critical points~$c$; cf.~\cite[Theorem~5.5.25]{SimonSzego}. We def\/ine the resolvent functions by
\begin{gather*}
r_-^J(z)=\big\langle (J_--z)^{-1}e_{-1},e_{-1}\big\rangle,\qquad	r^J_+(z)=\big\langle (J_+-z)^{-1}e_{0},e_{0}\big\rangle,\qquad z\in\bbC_+.
\end{gather*}
In the periodic case they can be given explicitly in terms of the orthogonal polynomials and they satisfy
\begin{gather}\label{def:reflectionless}
\frac{1}{r^J_+(x+i0)}=a_0^2\overline{r^J_-(x+i0)},\qquad \text{for almost all} \ \ x\in E.
\end{gather}
If the resolvent functions of a two-sided Jacobi matrix satisf\/ies~\eqref{def:reflectionless} on a set~$A$, we call it ref\/lectionless on~$A$. This property is characteristic in the following sense. Let us def\/ine for the given set $E=T_p^{-1}([-2,2])$ the isospectral torus (f\/inite gap class) of Jacobi matrices, $J(E)$, by
\begin{gather}\label{def:isospectral}
J(E)=\{J\colon \sigma(J)=E\text{ and } J \text{ is ref\/lectionless on }E\}.
\end{gather}
Then $J(E)$ consists of all periodic Jacobi matrices, whose spectrum is the set~$E$. The following well-known parametrization justif\/ies the name torus:
\begin{gather}\label{eq:parametrizationJacobi}
J(E)=\big\{J(\alpha)\colon \alpha\in\bbR^g/\bbZ^g\big\},
\end{gather}
where the map $\alpha\mapsto J(\alpha)$ is one-to-one; cf.~\cite{Akh60,Aptek84,Kri78,JacobiTeschl}. We recall a proof of this fact in Section~\ref{sec:fm} based on the previously mentioned second basis. The idea in this construction is the following: Let~$\Gamma^*$ be the group of all characters of the fundamental group of the domain $\overline{\bbC}\setminus E$. Note that~$\Gamma^*\cong \bbR^g/\bbZ^g$. Due to the properties of $T_p$, one can def\/ine a function~$\Phi(z)$ by
\begin{gather}\label{eq:MagicFormulafm}
	T_p(z)=\Phi(z)+\frac{1}{\Phi(z)},\qquad z\in \overline{\bbC}\setminus E.
\end{gather}
 This is possible since $T_p(z)\in \overline{\bbC}\setminus [-2,2]$ for $z\in \overline{\bbC}\setminus E$, which is the image of the Joukowski map $\zeta\mapsto\zeta+\frac{1}{\zeta}$. Fix $\alpha\in \Gamma^*$ and let $H^2(\alpha)$ be the Hardy space of character automorphic functions on $\overline{\bbC}\setminus E$ with character $\alpha$ and $L^2(\alpha)$ the corresponding space of character automorphic square integrable functions. The decomposition
\begin{gather*}
	H^2(\alpha)=K_\Phi(\alpha)\oplus\Phi H^2(\alpha)
\end{gather*}
def\/ines a natural basis $\{e_n^\alpha\}_{n=0}^\infty$ of $H^2(\alpha)$ and respectively $\{e_n^\alpha\}_{n=-\infty}^\infty$ a basis of $L^2(\alpha)$ with the properties that $e_{n+p}^\alpha=\Phi e_{n}^\alpha$ and $K_\Phi(\alpha)=\{e_0^\alpha,\dots,e_{p-1}^\alpha\}$. The multiplication by $z$ in the basis $\{e_n^\alpha\}_{n=-\infty}^\infty$ is a Jacobi matrix~$J(\alpha)$. Moreover, if $\alpha$ runs through~$\Gamma^*$ we obtain~$J(E)$, i.e., this construction proves~\eqref{eq:parametrizationJacobi}. Note that $e_{n+p}^\alpha=\Phi e_{n}^\alpha$ means that $\Phi$ is the symbol of $S^p$, where $S$ is the right shift on~$\ell^2$, i.e., $Se_n=e_{n+1}$. Since $z\Phi(z)=\Phi(z)z$ implies $J(\alpha)S^p=S^pJ(\alpha)$, the existence of the function $\Phi$ shows the periodicity of the Jacobi matrices~$J(\alpha)$. Finally, we would like to point out that the relation~\eqref{eq:MagicFormulafm} is the so-called magic formula in terms of the functional model, which is a surprising characterization of the isospectral torus of periodic Jacobi matrices. It states that
\begin{gather*}
	J\in J(E)\iff T_p(J)=S^p+S^{-p}.
\end{gather*}

We have seen that spectra of periodic Jacobi matrices are a f\/inite union of intervals of a very special structure. Namely, there has to be a polynomial $T_p$ with the properties~1),~2) and~3). In fact, it is not hard to show that this condition is also suf\/f\/icient. Nevertheless, the isospectral torus \eqref{def:isospectral} can be def\/ined for more general sets, in particular for so-called f\/inite gap sets,~$E$, of the form
\begin{gather}\label{def:spectralSet}
E=[\mathbf b_0,\mathbf a_0]\setminus\bigcup_{j=1}^g(\mathbf a_j,\mathbf b_j),\qquad g\in\bbN,\qquad \ba_j<\bb_j<\ba_{j+1}.
\end{gather}
The corresponding bases $\{e_n^\alpha\}$ can be def\/ined as well, which leads exactly in the same way to a~parametrization of $J(E)$ by $\Gamma^*$. Moreover, it gives explicit formulas for the coef\/f\/icients of $J(\alpha)$ by means of continuous functions $\cA$, $\cB$ on $\Gamma^*$, i.e.,
\begin{gather}\label{eq:coefficientsIsospectral}
a_n(\alpha)=\cA(\alpha-n\mu),\quad b_n(\alpha)=\cB(\alpha-n\mu),
\end{gather}
where $\mu$ is a f\/ixed character def\/ined only by the spectrum. This in particular implies that all elements of $J(E)$, which may not be periodic, are for sure almost periodic. Note that~\eqref{eq:coefficientsIsospectral} coincides with the formulae given in~\cite[Theorem~9.4]{JacobiTeschl}. In fact, this holds even for Jacobi mat\-rices with inf\/inite gap, homogeneous spectra. Hardy classes on f\/initely or inf\/initely connected domains are discussed, e.g., in~\cite{Has83,Heins69,Wid71}.

Since it is based on the existence of a polynomial, $T_p$, with the properties~1),~2) and~3), the characterization of the isospectral torus in terms of the magic formula is only possible for spectra of periodic Jacobi matrices, which turned out to be a powerful tool in the past. It was crucial in proving the f\/irst generalization of the remarkable Killip--Simon theorem; cf.~\cite{KillipSimon,KS03}. More specif\/ically, Damanik, Killip and Simon~\cite{KillipSimon} were able to generalize the Killip--Simon theorem for the case that $E$ is the spectrum of periodic Jacobi matrices and $|T_p(c)|>2$ for all critical points~$c$, which is a strong restriction on~$E$. The idea of GMP matrices is to substitute the polyno\-mial~$T_p$ by a~rational function. In~\cite{SMP}, we carried out this idea for the simplest case, namely if~$E$ is the arbitrary union of two distinct intervals. Note that the Damanik, Killip and Simon theorem only covers two intervals of equal length. Nevertheless, we can always f\/ind a~rational function,~$\Delta_E$, of the form
\begin{gather}\label{eq:DeltaOneInterval}
\Delta_E(z)=\lambda_0z+\bc_0+\frac{\lambda_1}{\bc_1-z},\qquad	\lambda_0,\lambda_1>0,
\end{gather}
such that $\Delta_E^{-1}([-2,2])=E$. By a linear change of variable we may assume that $\bc_1=0$. This suggests to consider matrices obtained by orthonormalizing the family of functions
\begin{gather}\label{eq:laurentpolynomials}
1,\,-\frac{1}{x},\,x,\,\frac{(-1)^2}{x^2},\,\dots,
\end{gather}
for a given real compactly supported measure~$\d\sigma_+$. Denoting this basis by~$\varphi_n$, we call the matrix of multiplication by the independent variable w.r.t.\ this basis a one-sided SMP matrix; see also~\cite{Kats16}. Let us mention that they are also called Jacobi-Laurent matrices; cf.~\cite{HeNi00}.

The connection to CMV matrices should not go unmentioned. CMV matrices are the Jacobi matrix analogue for measures supported on the unit circle. Already Szeg\H o discussed orthogonal polynomials,~$\psi_n$, w.r.t.\ a measure,~$\d\mu$, supported on the unit circle and showed that there are constants $\{\alpha_n\}_{n=0}^\infty$ in~$\bbD$, called Verblunsky coef\/f\/icients, so that
\begin{gather*}
	\sqrt{1-|\alpha_n|^2}\psi_{n+1}(z)=z\psi_n(z)-\overline{\alpha_n}z^n\overline{\psi_n(1/\overline{z})}.
\end{gather*}
 Due to Verblunsky~\cite{Verb35}, who def\/ined them in another context, the map $\d\mu\mapsto\{\alpha_n\}$ is one-to-one and onto all of~$\bbD^\infty$. Recent developments are due to Cantero, Moral and Vel\'azques~\cite{CanMoVel03}. They considered bases obtained by orthonormalizing families of the sort~\eqref{eq:laurentpolynomials} and showed that the matrix of the multiplication operator is a special structured f\/ive-diagonal matrix. For a given measure,~$\d\mu$, the entries can be given in terms of the Verblunsky coef\/f\/icients. Recognizing this characteristic structure they could use it to give a constructive def\/inition of CMV matrices, which uniquely def\/ined them, in the sense that there is a one-to-one correspondence between measures and CMV matrices. For a review on CMV matrices see~\cite{SimonCMVreview}.

In \cite{SMP}, we were also able to identify this characteristic structure for SMP matrices and to give a constructive def\/inition of them. Again, it was then more convenient for us to def\/ine them as two-sided matrices. Roughly speaking, a SMP matrix $A$ and its shifted inverse $-S^{-1}A^{-1}S$ (note that we assumed that $0$ is not in the spectrum) are f\/ive-diagonal matrices such that all even entries on the most outer diagonal vanish and the odd ones are positive. This structure perfectly f\/its to the following ``generalized magic formula'':
\begin{Proposition}[\cite{SMP}]
	Let $E$ be an arbitrary union of two intervals around zero and $\Delta_E$ the corresponding rational function of \eqref{eq:DeltaOneInterval}. Moreover, let $A(E)$ be the set of all two-periodic SMP matrices with its spectrum on $E$. Then
	\begin{gather*}
	A\in A(E) \ \gdw \ \Delta_E(A)=S^2+S^{-2}.
	\end{gather*}
\end{Proposition}

To deal with arbitrary f\/inite gap sets of the form~\eqref{def:spectralSet}, one f\/irst has to generalize~\eqref{eq:DeltaOneInterval}, which is done by the following lemma.

The Ahlfors function, $\Psi$, of the domain~$\overline{\bbC}\setminus E$ is the function that maximizes the value $\operatorname{Cap}_\text{a}(E)=\big|\lim\limits_{z\to\infty}z\Psi(z)\big|$ (the so-called analytic capacity), among all functions, which vanish at inf\/inity and are bounded by one in modulus.
\begin{Lemma}\label{lem:Delta}
	The function
	\begin{gather}\label{eq:DeltaAhlfors}
	\Delta_E(z):=\Psi(z)+\frac{1}{\Psi(z)}
	\end{gather}
	is a rational function of the form
	\begin{gather}\label{eq:14}
	\Delta_E(z)=\lambda_0z+\bc_0+\sum_{j=1}^g\frac{\lambda_j}{\bc_j-z},
	\end{gather}
	with $\lambda_j>0$, $j\ge 0$, $\bc_j\in(\ba_j,\bb_j)$, $j\ge 1$ and
	\begin{gather}\label{eq13}
	E=[\bb_0,\ba_0]\setminus\bigcup_{j=1}^g(\ba_j,\bb_j)=\Delta_E^{-1}([-2,2]).
	\end{gather}
	In fact, if we demand that $\Im \Delta_E(z)>0$ for $\Im z>0$ and $\lim\limits_{z\to\infty}\Delta_E(z)=\infty$, $\Delta_E$ is the unique rational function with the property \eqref{eq13}.
\end{Lemma}
\begin{proof}
	Due to \cite{Pom1}, we have
	\begin{gather*}
	\frac{1-\Psi(z)}{1+\Psi(z)}=\sqrt{\prod_{j=0}^{g}\frac{z-\ba_j}{z-\bb_j}}=:G(z).
	\end{gather*}
	Therefore,
	\begin{gather*}
	\Delta_E(z):=\Psi(z)+\frac{1}{\Psi(z)}=2\frac{1+G^2(z)}{1-G^2(z)},
	\end{gather*}
	is of the form $P_{g+1}(z)/Q_m(z)$, where $m\leq g$. Like in \cite[Chapter VII]{Lev80}, we see that $\Im G^2(z)>0$ for $\Im z>0$, which then clearly also holds for $\Delta_E$. Therefore, $G^2$ is increasing on the interval $(\ba_j,\bb_j)$ and has a zero at $\ba_j$ and a pole at $\bb_j$. Hence, there is exactly one pole of $\Delta_E$ in each gap. Since $G^2(\infty)=1$, there is also a pole at inf\/inity.
	To prove uniqueness let $\Delta(z)$ be a rational function with the claimed properties. First, we notice that, due to the argument principle, it is not possible that $\Delta$ has a pole in the upper half plane. This and $\lim\limits_{z\to\infty}\Delta(z)=\infty$ already implies that $\Delta$ is of the form~\eqref{eq:14}. Since $\Delta'(x)>0$ on $\bbR\setminus\{\bc_1,\dots,\bc_g\}$, $\Delta^{-1}([-2,2])=E$ implies
	\begin{gather*}
	\Delta(\ba_j)=2 \qquad\text{and}\qquad \Delta(\bb_j)=-2\qquad \text{for} \ \ j\geq 0.
	\end{gather*}
	This def\/ines $\Delta$ uniquely.
\end{proof}

Let $\bC=\{\bc_1,\dots,\bc_g\}$ be a collection of distinct real points and $\d\sigma_-$ a measure such that the points $\bc_k$ don't belong to its support. Like SMP matrices, we def\/ine a one-sided GMP mat\-rix,~$A_-$, as the matrix of the multiplication operator w.r.t.\ the basis obtained by orthonormalizing the family of functions
\begin{gather*}
1,\,\frac{1}{\bc_g-x},\,\frac{1}{\bc_{g-1}-x},\,\dots, \,\frac{1}{\bc_1-x},\,x,\,\frac{1}{(\bc_g-x)^2},\,\dots.
\end{gather*}
This is discussed in detail in the Appendix of~\cite{YudKillipSimon}; see also~\cite{KupBar11}. Their characteristic structure, which looks quite complicated at the f\/irst glance, will be used in the following def\/inition, but f\/irst we would like to point out another property of multiplication operators w.r.t.\ rational functions. Namely, since~$\bc_k$ does not belong to the support of the measure, we can also consider multiplication by~$\frac{1}{\bc_k-x}$. Hence, if the f\/irst block of~$A_-$ corresponds to the basis
\begin{gather}\label{eq:gmpcyclic}
\left[1,\frac{1}{\bc_1-x},\dots,\frac{1}{\bc_g-x}\right]
\end{gather}
then the linear change of variable $y=\frac{1}{\bc_k-x}$ leads to the basis related to
\begin{gather}\label{eq:gmpcyclicShifted}
\left[-\frac{1}{y},\frac{1}{y(\bc_1)-y},\dots,1,\dots,\frac{1}{y(\bc_g)-y} \right],
\end{gather}
which says that the shifted resolvents should be of the same shape. In fact, in the construction the spaces~\eqref{eq:gmpcyclic} and~\eqref{eq:gmpcyclicShifted} serve as cyclic subspaces for $\Delta(x)$ and $\tilde \Delta(y)=\Delta(x)$, respectively. See also proof of Theorem~\ref{prop:isospectralGMP}. The structure of $A_-$ and this certain invariance property of the resolvents is now used as a def\/inition for two-sided GMP matrices.

By $T^*$ we denote the conjugated operator to an operator $T$, or the conjugated matrix if $T$ is a matrix. In particular, for a column vector $\vp\in \bbC^{g+1}$, $(\vp)^*$ is a $(g+1)$-dimensional row vector. The notation $T^-$ denotes the upper triangular part of a matrix $T$ (excluding the main diagonal), and $T^+=T-T^-$ is its lower triangular part (including the main diagonal).

First of all, the GMP class depends on an ordered collection of distinct points $\bC\!=\!\{\bc_1,{\dots},\bc_g\}$.
\begin{Definition}
	We say that $A$ is of the class $\bbA$ if it is a $(g+1)$-block Jacobi matrix
	\begin{gather*}
	A=\begin{bmatrix}
	\ddots&\ddots&\ddots&& &\\
	&A^*(\vp_{-1})&B(\vbp_{-1})&A(\vp_0)& & \\
	& &A^*(\vp_{0})&B(\vbp_{0})&A(\vp_1)& \\
	& & &\ddots&\ddots&\ddots
	\end{bmatrix}
	\end{gather*}
	such that
	\begin{gather*}
	\vbp=(\vp,\vq\,)\in\bbR^{2g+2},\qquad A(\vp)=\delta_g \vp\,^*,
	\qquad
	B(\vbp)
	=(\vq \vp\,^*)^-+(\vp\vq\,^*)^++\tilde\bC,
	\end{gather*}
	and
	\begin{gather*}
	\tilde \bC=\begin{bmatrix}
	\bc_1& & & \\
	& \ddots& & \\
	& & \bc_g & \\
	& & &0
	\end{bmatrix},\qquad
	\vp_j=
	\begin{bmatrix}
	p^{(j)}_0\\
	\vdots\\
	p^{(j)}_g
	\end{bmatrix}, \qquad
	\vq_j=
	\begin{bmatrix}
	q^{(j)}_0\\
	\vdots\\
	q^{(j)}_g
	\end{bmatrix}, \qquad p^{(j)}_g>0.
	\end{gather*}
	We call $\{\vbp_j\}_{j\in\bbZ}$ the generating coef\/f\/icient sequences (for the given~$A$).
\end{Definition}
\begin{Definition}\label{def:GMP}
	A matrix $A\in\bbA$ belongs to the GMP class if the matrices $\{\bc_k-A\}_{k=1}^g$, for $1\leq k\leq g$, are invertible, and moreover
	$S^{-k}(\bc_k-A)^{-1}S^k$ are also of the class $\bbA$. To abbreviate we write $A\in \GMP(\bC)$.
\end{Definition}
We call a GMP matrix one-block periodic or simply periodic if $\vbp_j=\vbp$ for all $j\in\bbZ$.

Thus, we pay a quite high price in giving up the simple structure of Jacobi matrices, but in return we get~\eqref{eq:DeltaAhlfors}, which will in particular allow us to parametrize the f\/inite gap class of almost periodic Jacobi matrices by periodic GMP matrices. In the same way as for Jacobi matrices, the decomposition
\begin{gather*}
	H^2(\alpha)=K_{\Psi}(\alpha)\oplus\Psi H^2(\alpha),
\end{gather*}
leads to a new basis $\{f_n^\alpha\}_{n=-\infty}^\infty$ such that $f_{n+p}^{\alpha}=\Psi f_n^\alpha$, $K_\Psi(\alpha)=\{f_0^\alpha,\dots,f_g^\alpha\}$ and the multiplication by $z$ is a~GMP matrix. Note that the property $f_{n+p}^{\alpha}=\Psi f_n^\alpha$ shows that the corresponding matrix is periodic. Def\/ining the isospectral torus of GMP matrices by
\begin{gather*}
A(E,\bC)=\{A\in \GMP(\bC)\colon \sigma(A)=E \text{ and }A\text{ is periodic} \},
\end{gather*}
we obtain the following analogue of \eqref{eq:parametrizationJacobi}:
\begin{Theorem}\label{prop:isospectralGMP}
	Let $E$ be a finite gap set and $\bC$ be a fixed ordering of the zeros of the corresponding Ahlfors function. Then
	\begin{gather*}
	A(E,\bC)=\big\{A(\alpha,\bC)\colon \alpha\in\bbR^g/\bbZ^g\big\},
	\end{gather*}
	where $A(\alpha,\bC)$ is the multiplication by the independent variable w.r.t.\ the basis $\{f_n^\alpha\}$. The map $\alpha\mapsto A(\alpha,\bC)$ is one-to-one up to the identification $(p_j,q_j)\mapsto(-p_j,-q_j)$ for $j=0,\dots, g-1$.
\end{Theorem}

Moreover, \eqref{eq:DeltaAhlfors} is the magic formula for GMP matrices in terms of our functional model.

\begin{Theorem}\label{thm:magicFormula}
	Let $A\in\GMP(\bC)$. Then
	\begin{gather*}
	A\in A(E,\bC) \iff \Delta_E(A)=S^{g+1}+S^{-(g+1)}.
	\end{gather*}
\end{Theorem}

This was one of the main observations which allowed Yuditskii in \cite{YudKillipSimon} to generalize the Killip--Simon theorem to arbitrary systems of intervals.

Finally, we would like to point out that the direct spectral theory of periodic GMP matrices has numerous similarities to the one of periodic Jacobi matrices. It is based on the fact that like Jacobi matrices GMP matrices can be written as a two-dimensional perturbation of a block diagonal matrix. Let $a_0=\|\vp\|$, $\tilde e_0=\frac{1}{a_0}P_+Ae_{-1}$ and $A_\pm=P_\pm AP_\pm^*$, $P_+=I-P_-$. Then
\begin{gather*}
A=\begin{bmatrix}
A_-&0\\
0&A_+
\end{bmatrix}
+
a_0(\langle \cdot, e_{-1}\rangle\tilde e_0+\langle \cdot, \tilde e_0\rangle e_{-1}).
\end{gather*}
\begin{Definition} Let $A\in\GMP(\bC)$ be a periodic GMP matrix with coef\/f\/icients $\vp$ and $\vq$. Let
	$
	\bp=\begin{bmatrix}
	p& q
	\end{bmatrix}\in\bbR^2.
	$
	We introduce the matrix functions
	\begin{gather}\label{thematrix}
	\fa(z;\bp)=\fa(z,\infty;p,q)=
	\begin{bmatrix}
	0&-{p}\\
	\frac 1{p}&\frac{z-pq}{p}
	\end{bmatrix},
	\end{gather}
	and
	\begin{gather*}
	\fa(z,\bc;\bp)=I-\frac{1}{\bc-z}
	\begin{bmatrix}
	p\\ q
	\end{bmatrix}\begin{bmatrix}
	p&q
	\end{bmatrix}\fj,\qquad
	\fj=
	\begin{bmatrix}
	0&-1\\
	1&0
	\end{bmatrix}.
	\end{gather*}
	Then the product
	\begin{gather*}
	\fA(z)=\fa(z,\bc_1;\bp_0)\fa(z,\bc_2;\bp_1)\cdots\fa(z,\bc_g;\bp_{g-1})\fa(z;\bp_g)
	\end{gather*}
	is called the transfer matrix associated with the given $A$. Moreover, we def\/ine its discriminant by
	\begin{gather}\label{def:discriminant}
	\Delta^A(z)=\tr \fA(z).
	\end{gather}
\end{Definition}
\begin{Theorem}\label{thm:spectrumOfperiodicA}
	Let $A\in\GMP(\bC)$ be a periodic GMP matrix with coefficients $\vp$ and $\vq$. Then $A$ has purely absolutely continuous spectrum, which is given by
	\begin{gather*}
	\sigma(A)=\sigma_{ac}(A)=\big\{z\in\bbC\colon \Delta^A(z)\in[-2,2]\big\}.
	\end{gather*}
\end{Theorem}
If for a given set $E$ $A\in A(E,\bC)$, then we show in Lemma~\ref{lem:deltaAequalDelta} that indeed $\Delta_E=\Delta^A$. This allows us to explain an alternative def\/inition of $\GMP$ matrices. We def\/ine for a periodic $\GMP$ matrix with generating coef\/f\/icients~$\vp$ and~$\vq$
\begin{gather}
\Lambda_{k}(\vbp)=
-\tr\left\{\prod_{m=0}^{k-2}\fa(\bc_k,\bc_{m+1};\bp_m)
\bp_{k-1} \bp_{k-1}^*\fj \prod_{m=k}^{g-1}\fa(\bc_k,\bc_{m+1};\bp_m) \fa(\bc_k;\bp_g)\right\},
\label{eq:defLambdaK}
\end{gather}
where by def\/inition $\Lambda_{k}(\vbp)=-\Res_{\bc_k}\tr\fA(z)$.

Let us consider the last non-zero diagonal of $\Delta_E(A)$, i.e., $\Delta_E(A)_{j,g+1+j}$ for $0\leq j\leq g$. By def\/inition of $\GMP$ matrices, $(\bc_k-A)^{-1}_{k-1,g+k}>0$ for $1\leq k\leq g$, whereas $(\bc_l-A)^{-1}_{k-1,g+k}=A_{k-1,g+k}=0$ for $l\neq k$. Moreover, $(\bc_k-A)^{-1}_{g,2g+1}=0$ for $1\leq k\leq g$ and $A_{g,2g+1}>0$. Thus, on the last non-zero diagonal of $\Delta_E(A)$ only one of the summands is non-zero. Note that the relation $\Delta_E(z)=\Delta^A(z)$ implies
$
\lambda_k=\Lambda_{k}(\vbp).
$
The previous consideration, the magic formula and this identity yield
\begin{gather*}
\Lambda_k(\vbp)(\bc_k-A)^{-1}_{k-1,g+k}=1 \qquad \text{for} \quad 1\leq k\leq g.
\end{gather*}
It is important to mention that all this just served as explanation, but is not necessary to prove the following alternative def\/inition of $\GMP$ matrices.

\begin{Theorem}
	Let $A\in \bbA$ be periodic with generating coefficients $\vbp$. Then $A\in\GMP(\bC)$ if and only if
	\begin{gather*}
	\Lambda_k(\vbp)>0\qquad \text{for} \quad 1\leq k\leq g.
	\end{gather*}
\end{Theorem}
\begin{proof}
	The proof is based on the idea that one can f\/ind the entries of the inverse matrices $(\bc_k-A)^{-1}$ explicitly; cf.~\cite[Lemma~3.2 and Theorem~3.3]{YudKillipSimon}
\end{proof}

The relation $\Delta_E(z)=\Delta^A(z)$ f\/inally leads to an algebraic description for $A(E,\bC)$.

\begin{Theorem}\label{cor:IsospectralAlgebraicManifold}
	Let $A\in\GMP(\bC)$ be periodic with generating coefficients $\vbp$. Then $A\in A(E,\bC)$ if and only if
	\begin{gather}\label{iso101}
	p_g=\frac 1{\lambda_0}, \quad q_g=-\bc_0-\lambda_0\sum_{j=1}^{g-1} p_jq_j, \qquad \Lambda_k(\vbp)=\lambda_k\qquad \text{for} \quad k=1,\dots ,g,
	\end{gather}
	where $\Lambda_k(\vbp)$ is defined as in~\eqref{eq:defLambdaK}.
\end{Theorem}

The organization of the paper is as follows. In Section~\ref{sec:directSpectral} we deal with the direct spectral theory of periodic GMP matrices. That is, we prove Theorem \ref{thm:spectrumOfperiodicA}. In Section~\ref{sec:fm} we show that periodic GMP matrices arise as the multiplication by the independent variable w.r.t.~$\{f_n^\alpha\}$ and prove Theorems~\ref{thm:magicFormula} and~\ref{cor:IsospectralAlgebraicManifold}. In both sections we f\/irst recall the known theory for Jacobi matrices and then adapt this construction to the GMP case.

The results of the paper were f\/irst announced in \cite{YudEichKillipSimon}.

\section{Direct spectral theory of periodic GMP matrices}\label{sec:directSpectral}
Let $J$ be a $p$-periodic two-sided Jacobi matrix with generating coef\/f\/icients $\{a_j,b_j\}_{j=0}^g$. Its transfer matrix is def\/ined by
\begin{gather*}
	\fA^J(z)=\fa(z,a_1,b_0/a_1)\fa(z,a_2,b_1/a_2)\cdots\fa(z,a_p,b_{p-1}/a_p),
\end{gather*}
where
\begin{gather*}
	\fa(z,a_j,b_{j-1}/a_j)=
	\begin{bmatrix}
	0&-{a_j}\\
	\frac 1{a_j}&\frac{z-b_{j-1}}{a_j}
	\end{bmatrix},
\end{gather*}
is def\/ined as in \eqref{thematrix}. Moreover, the discriminant is given by $T_p(z)=\tr\fA^J(z)$. This is not a~notational conf\/lict, but $\tr\fA^J(z)$ is indeed the polynomial in~\eqref{eq:PolynomialJacobi}.
The spectrum of $J$ is purely absolutely continuous and
\begin{gather}\label{eq:spectrumJ}
	\sigma(J)=\sigma_{\text{ac}}(J)=T_p^{-1}([-2,2]),
\end{gather}
cf.~\cite[Chapter~5]{SimonSzego}. One way of proving this is to write $J$ as
\begin{gather}\label{eq:orthoDecompJ}
	J=\begin{bmatrix}J_-& 0\\
	0& J_+
	\end{bmatrix}
	+
	a_0(\langle\cdot,e_{-1}\rangle e_0+\langle\cdot,e_{0}\rangle e_{-1}),
\end{gather}
where $\{e_0,e_{-1}\}$ spans a cyclic subspace for $J$. From this representation it is easy to deduce that the matrix resolvent function $R^J(z)$, def\/ined by
\begin{gather*}
	R^J(z)=\begin{bmatrix}
		\langle (J-z)^{-1}e_{-1},e_{-1}\rangle& \langle (J-z)^{-1}e_{0},e_{-1}\rangle\\
		\langle (J-z)^{-1}e_{-1},e_{0}\rangle&\langle (J-z)^{-1}e_{0},e_{0}\rangle
	\end{bmatrix},
\end{gather*}
admits the representation
\begin{gather*}
	R^J(z)=\begin{bmatrix}
	r^J_{-}(z)^{-1}& a_0\\
	a_0& r^J_+(z)^{-1}
	\end{bmatrix}^{-1},
\end{gather*}
where $r^J_{\pm}(z)=\langle(J_\pm-z)^{-1}e_{\frac{-1\pm 1}{2}},e_{\frac{-1\pm 1}{2}}\rangle$.
Using what is called coef\/f\/icient stripping in \cite[Theorem~3.2.4]{SimonSzego} one can then f\/ind a representation for $R^J(z)$ that implies \eqref{eq:spectrumJ}. In this chapter we will follow this strategy for $\GMP$ matrices.

Let $A$ be a one-block periodic $\GMP$ matrix. Due to \cite[Proposition~5.5]{YudKillipSimon}, $\{e_{-1},\tilde e_0\}$ span a~cyclic subspace for~$A$. Like in~\eqref{eq:orthoDecompJ}, we can represent~$A$ as
\begin{gather*}
	A=\begin{bmatrix}A_-& 0\\
	0& A_+
	\end{bmatrix}
	+
	a_0(\langle\cdot,e_{-1}\rangle \tilde e_0+\langle\cdot,\tilde e_{0}\rangle e_{-1}),
\end{gather*}
where $a_0=\|\vp\,\|$. Hence, def\/ining
\begin{gather*}
	R(z)=\begin{bmatrix}
	\langle (A-z)^{-1}e_{-1},e_{-1}\rangle& \langle (A-z)^{-1}\tilde e_{0},e_{-1}\rangle\\
	\langle (A-z)^{-1}e_{-1},\tilde e_{0}\rangle&\langle (A-z)^{-1}\tilde e_{0},\tilde e_{0}\rangle
	\end{bmatrix},
\end{gather*}
we obtain
\begin{gather}\label{eq:ResolventGMP}
	R(z)=\begin{bmatrix}
	r_{-}(z)^{-1}& a_0\\
	a_0& r_+(z)^{-1}
	\end{bmatrix}^{-1},
\end{gather}
where
\begin{gather*}
	r_-(z)=\langle (A_--z)^{-1}e_{-1},e_{-1}\rangle,\qquad r_+(z)=\langle (A_+-z)^{-1}\tilde e_{0},\tilde e_{0}\rangle.
\end{gather*}
The following theorem is an analogue of \cite[Theorem 3.2.4]{KillipSimon} for periodic GMP matrices.

First, we introduce some notations, which will be used in the proof. For $\vec{x}\in\bbR^{g+1}$, we def\/ine
\begin{gather*}
s_k\vec x=\begin{bmatrix} x_0\\ \vdots \\ x_{g-k} \end{bmatrix}.
\end{gather*}
Moreover, let the $M_k$'s be upper triangular matrices such that
\begin{gather*}
B(\vp,\vq\,)-\vp(\vq\,)^*=M(\vp,\vq\,):=M_0=\begin{bmatrix}M_1& 0\\
0 & 0\end{bmatrix}+(-\vp\, q_g+\vq\, p_g)\delta_g^*
\end{gather*}
and
\begin{gather*}
M_k=\begin{bmatrix}M_{k+1}& 0\\
0 & \bc_{g+1-k}\end{bmatrix}+(-s_k\vp\, q_{g-k}+{s_k\vq}\, p_{g-k})(s_k\delta_{g-k})^*\qquad \text{for} \quad 1\leq k\leq g-1.
\end{gather*}
\begin{Theorem}\label{thm:transfermatrix1}
	Let
	\begin{gather*}
	\begin{bmatrix}
	R(z,p,p)&R(z,g,p)\\
	R(z,p,g)& R(z,g,g)
	\end{bmatrix}=
	\begin{bmatrix}
	\langle(B(\vp,\vq\,)-z)^{-1}\vp,\vp\,\rangle& \langle(B(\vp,\vq\,)-z)^{-1}\vec\delta_g,\vp\,\rangle \\
	\langle(B(\vp,\vq\,)-z)^{-1}\vp, \vec\delta_g\rangle& \langle(B(\vp,\vq\,)-z)^{-1}\vec\delta_g,\vec\delta_g\rangle
	\end{bmatrix}.
	\end{gather*}
	Let $A$ be a periodic GMP matrix, $r_+(z)$ the resolvent function of $A_+$, i.e.,
	\begin{gather*}
	r_+(z)=\big\langle(A_+-z)^{-1}\tilde e_0,\tilde e_0\big\rangle.
	\end{gather*}
	Then we have
	\begin{gather}\label{eq:transfermatrix1}
	a^2_0r_+(z)=\frac{a^2_0r_+(z)\tilde\fA_{11}(z)+\tilde\fA_{12}(z)}{a^2_0r_+(z)\tilde\fA_{21}(z)+\tilde\fA_{22}(z)},
	\end{gather}
	where $a^2_0=\|\vp\|^2$ and
	\begin{gather*}
	\tilde\fA(z)=
	\begin{bmatrix}
	\tilde\fA_{11}(z)&\tilde\fA_{12}(z)\\
	\tilde\fA_{21}(z)&\tilde\fA_{22}(z)
	\end{bmatrix} 	=
	\frac{1}{R(z,p,g)}
	\begin{bmatrix}
	R(z,p,p)R(z,g,g)-R(z,p,g)^2& -R(z,p,p)\\
	R(z,g,g)&-1
	\end{bmatrix}.
	\end{gather*}
\end{Theorem}
\begin{proof}
	We write
	\begin{gather*}
		A_+=
		\begin{bmatrix}
		B(\vp,\vq\,)& 0\\
		0& A_+
		\end{bmatrix}+
		a_0(\langle\cdot,e_{g}\rangle\tilde e_1+\langle \cdot,\tilde e_1\rangle e_g),
	\end{gather*}
	where $\tilde e_1=S^{g+1}\tilde e_0$ and apply the Sherman--Morrison--Woodbury formula (cf.~\cite[Section~2.1.3]{GoVL13}) to prove the theorem.
\end{proof}

\begin{Theorem}
	Let $\tilde\fA$ be defined as in Theorem~{\rm \ref{thm:transfermatrix1}} and $\fA$ be the transfer matrix of $A$. Then we have $\tilde\fA(z)=\fA(z)$.
\end{Theorem}
\begin{proof}
	First, we represent $B(\vp,\vq)$ as a one-dimensional perturbation of a lower diagonal matrix. Applying the Sherman--Morrison--Woodbury formula again, leads to a representation of $\tilde\fA$ in terms of~$M_0$. Using that~$M_0$ is a~lower triangular matrix, we obtain
	\begin{gather*}
	\tilde \fA(z)=\fA_0(z)\fa(z;p_g,q_g),
	\end{gather*}
	where
	\begin{gather*}
	\fA_0(z)=I-\begin{bmatrix} \langle (M_1-z)^{-1} {u_1\vp},{u_1\vp}\,\rangle&
	\langle (M_1-z)^{-1}{u_1\vq},{u_1\vp}\,\rangle\\
	\langle (M_1-z)^{-1}{u_1\vp},{u_1\vq}\,\rangle & \langle (M_1-z)^{-1}{u_1\vq},{u_1\vq}\,\rangle
	\end{bmatrix}\fj.
	\end{gather*}
	Using again that all $M_j's$ are lower triangular matrices, we f\/ind that
	\begin{gather*}
	\fA_{j-1}(z)=\fA_{j}(z)\fa(z,\bc_{g+1-j};p_{g-j},q_{g-j}),
	\end{gather*}
	where
	\begin{gather*}
	\fA_{j-1}(z)=I-\begin{bmatrix} \langle (M_j-z)^{-1}{u_j\vp},{u_j\vq}\,\rangle&
	\langle (M_j-z)^{-1}{u_j\vq},{u_j\vp}\,\rangle\\
	\langle (M_j-z)^{-1}{u_j\vp},{u_j\vq}\,\rangle & \langle (M_j-z)^{-1}{u_j\vq},{u_j\vq}\,\rangle
	\end{bmatrix}\fj.\tag*{\qed}
	\end{gather*}
\renewcommand{\qed}{}
\end{proof}
\begin{Remark}\quad
	\begin{enumerate}\itemsep=0pt
		\item Note that $\fA$ is normalized such that $\det\fA=1$.
		\item For a general (non periodic) $\GMP$ matrix the only dif\/ference in~\eqref{eq:transfermatrix1} is that in the right-hand side $a_0$ is replaced by $\|\vp_1\|$ and $r_+$ by $r_+^{(1)}$, which is the resolvent function related to the shifted $\GMP$ matrix $S^{-(g+1)}AS^{g+1}$. This explains the name transfer matrix.
		\item Unlike Jacobi matrices, we consider the relation between the resolvent function of the initial and the $g+1$-shifted $\GMP$ matrix, since this shift preserves its structure.
	\end{enumerate}
\end{Remark}
Applying the same calculations to $r_-$ leads to the following theorem.
\begin{Theorem}\label{thm:reflextranfermatrix}
	For a periodic GMP matrix $A$, let
	\begin{gather*}
		\fA_-(z)=\begin{bmatrix}
		\fA^-_{11}(z)& \fA^-_{12}(z)\\
		\fA^-_{21}(z)& \fA^-_{22}(z)
		\end{bmatrix}
	\end{gather*}
 be the transfer matrix of $r_-$, i.e.,
	\begin{gather}\label{eq:transferMatrixMinus}
	r_-(z)=\frac{r_-(z)\fA^-_{11}(z)+\fA^-_{12}(z)}{r_-(z)\fA^-_{21}(z)+\fA^-_{22}(z)}.
	\end{gather}
	Then it is of the form
	\begin{gather*}
	\fA_-(z)=\fa_-(z,\infty;p_g,q_g)\fa_-(z,c_g;p_{g-1},q_{g-1})\cdots\fa_-(z,c_1;p_0,q_0),
	\end{gather*}
	where
	\begin{gather*}
	\fa_-(z,\infty;p,q)=\begin{bmatrix}
	0&-\frac{1}{p}\\
	p&\frac{z-pq}{p}
	\end{bmatrix}\qquad\text{and}\qquad
	\fa_-(z,\bc;p,q)=I-\frac{1}{\bc-z}
	\begin{bmatrix}
	q\\ p
	\end{bmatrix}\begin{bmatrix}
	q&p
	\end{bmatrix}\fj.
	\end{gather*}
\end{Theorem}

\begin{proof}
	The proof is the same as for $r_+$, but in order to extract $\fa_-(z,\infty;p,q)$ in the f\/irst step, one has to write the ``mirrored'' $B$-block as a~one-dimensional perturbation of an upper triangular matrix.
\end{proof}

The following corollary will be an important ingredient in the proof of Theorem \ref{thm:spectrumOfperiodicA}.
\begin{Corollary}\label{cor:TransferRelation}
	With the notation from above the entries of $\fA$ and $\fA_-$ are related by
		\begin{gather*}
		\fA_{11}(z)=\fA_{11}^-(z),\qquad \fA_{12}(z)=-\fA_{21}^-(z),\qquad \fA_{21}(z)=-\fA_{12}^-(z),\qquad \fA_{22}(z)=\fA_{22}^-(z).
		\end{gather*}
	\end{Corollary}
	\begin{proof}
		This follows by the relation
		\begin{gather*}
		\fA(z)=\begin{bmatrix}
		1&0\\
		0&-1
		\end{bmatrix}
		\fA_-(\overline{z})^*
		\begin{bmatrix}
		1&0\\
		0&-1
		\end{bmatrix}.\tag*{\qed}
		\end{gather*}
\renewcommand{\qed}{}
	\end{proof}

Clearly, \eqref{eq:transfermatrix1} is a quadratic equation for $a_0^2r_+$. Let $\Delta^A(z)=\tr\fA(z)$ and $V(z)=\fA_{11}(z)-\fA_{22}(z)$. Using the normalization of $\fA$, we see that for all $z\in\bbC_+$
	\begin{gather}\label{eq:resolventplus}
	a_0^2r_+(z)=\frac{1}{2\fA_{21}(z)}\Big(V(z)+\sqrt{\Delta^A(z)^2-4}\Big),
	\end{gather}
	where one takes the branch of the square root with $\sqrt{\Delta^A(z)^2-4}=\Delta^A(z)+\mathcal{O}\big(\frac{1}{\Delta^A(z)}\big)$ near $z=\infty$.
	\begin{Lemma}\label{lem:rminusSecondSolution}
		The function $r_-^{-1}$ is the second solution of \eqref{eq:transfermatrix1}. That is,
		\begin{gather*}
			\frac{1}{r_-(z)}=\frac{1}{2\fA_{21}(z)}\Big(V(z)-\sqrt{\Delta^A(z)^2-4}\Big),
		\end{gather*}
	\end{Lemma}
	\begin{proof}
		Due to \eqref{eq:transferMatrixMinus}, we have
		\begin{gather*}
			r_-(z)=\frac{\fA_{22}^-r_--\fA_{12}^-}{-\fA_{21}^-r_-+\fA_{11}^-}(z).
		\end{gather*}
		Using Corollary \ref{cor:TransferRelation}, we see that $r_-^{-1}$ is a solution of~\eqref{eq:transfermatrix1}. That $r_-^{-1}$ is distinct form~$a_0^2r_+$ on~$\bbC_+$ follows since $\Im r_-^{-1}<0$, while $\Im a_0^2r_+>0$ on $\bbC_+$.
	\end{proof}
	\begin{Lemma}
		The resolvent function of a periodic $\GMP$ matrix $A$ admits the following representation:
		\begin{gather*}
			R(z)=\frac{1}{2a_0^2\sqrt{\Delta^A(z)^2-4}}
			\begin{bmatrix}
				-2a_0^2\fA_{21}(z)& a_0V(z)\\
				a_0V(z)& 2\fA_{12}(z)
			\end{bmatrix}
			+\frac{1}{2a_0}
			\begin{bmatrix}
			0& 1\\
			1& 0
			\end{bmatrix}.
		\end{gather*}
	\end{Lemma}
	\begin{proof}
		This is a consequence of \eqref{eq:ResolventGMP} and Lemma \ref{lem:rminusSecondSolution}.
	\end{proof}
	\begin{proof}[Proof of Theorem \ref{thm:spectrumOfperiodicA}]
		We have $\sigma_{ac}(A)=\{z\in\bbC\colon \Delta^A(z)\in[-2,2]\}$. Pure point spectrum can only appear at poles of $\fA_{ij}$. But since $-\Res_{\bc_k}\Delta^A(z)>0$, this is not possible.
	\end{proof}

\section{Functional models}\label{sec:fm}
\subsection{Def\/initions}
We start with a uniformization of the domain $\overline{\bbC}\setminus E$. That is, there exists a Fuchsian group $\Gamma$ and a meromorphic function $\fz\colon \bbD\to\overline{\bbC}\setminus E$ with the following properties:
\begin{itemize}\itemsep=0pt
	\item[(i)]$\forall\, z\in\overline{\bbC}\setminus E\ \exists\,\zeta\in\bbD\colon \fz(\zeta)=z$,
	\item[(ii)]$\fz(\zeta_1)=\fz(\zeta_2)\ \gdw\ \exists\,\gamma\in\Gamma\colon \zeta_1=\gamma(\zeta_2)$.
\end{itemize}
We f\/ix $\fz$ by the normalization condition $\fz(0)=\infty$ and $\lim\limits_{\zeta\to 0}(\zeta\fz)(\zeta)>0.$ By
\begin{gather*}
\Gamma^*=\{\alpha\,|\, \alpha\colon\Gamma\to \bbR/\bbZ \text{ such that } \alpha(\gamma_1\gamma_2)=\alpha(\gamma_1)+\alpha(\gamma_2)\}
\end{gather*}
we denote the group of characters of $\Gamma$. Note that $\Gamma$ is equivalent to the fundamental group of~$\overline{\bbC}\setminus E$ and therefore $\Gamma^*\cong \bbR^g/\bbZ^g$. Let $H^2=H^2(\bbD)$ denote the standard Hardy space of the disk. For $\alpha\in\Gamma^*$ we def\/ine the Hardy space of character automorphic functions by
\begin{gather*}
	H^2(\alpha)=\big\{f\in H^2\colon f\circ\gamma=e^{2\pi i \alpha(\gamma)}f,\, \gamma\in\Gamma\big\}.
\end{gather*}
Fix $z_0\in\overline{\bbC}\setminus E$ and consider the associated orbit $\operatorname{orb}(\zeta_0)=\fz^{-1}(z_0)=\{\gamma(\zeta_0)\colon \gamma\in\Gamma\}$. The Blaschke product $b_{z_0}(\zeta)$ with zeros at $\fz^{-1}(z_0)$ is called the Green function of the group~$\Gamma$; cf.~\cite{Pom76}. It is normalized so that $b_{z_0}(0)> 0$ if $z_0\neq \infty$ and $(\fz b_{\infty})(0)>0$. It is related with the standard Green function $G(z,z_0)$ of $\overline{\bbC}\setminus E$ by
\begin{gather*}
\log\frac{1}{|b_{z_0}(\zeta)|}=G(\fz(\zeta),z_0).
\end{gather*}
The function $b_{z_0}$ is character automorphic. We denote the corresponding character by $\mu_{z_0}$. We def\/ine $k^\alpha(\zeta,\zeta_0)=k^\alpha_{\zeta_0}(\zeta)$ as the reproducing kernels of $H^2(\alpha)$, i.e.,
\begin{gather*}
\langle f, k^\alpha_{\zeta_0} \rangle=f(\zeta_0),\qquad\forall\, f\in H^2(\alpha).
\end{gather*}
If $z_0=\infty$, we use the abbreviations $b(\zeta)=b_\infty(\zeta)$, $\mu=\mu_\infty$ and $k^\alpha(\zeta)=k^\alpha_0(\zeta)$.

\subsection{Functional models for Jacobi matrices}
In this section we recall some results from the theory of f\/inite gap Jacobi matrices. In particular, we will use functional models to explain the appearance of a polynomial $T_p$ in the theory of periodic Jacobi matrices.
\begin{Theorem}[\cite{KontaniLast}]	\label{thm:onb}
	The system of functions
	\begin{gather*}
	e_{n}^{\alpha}(\zeta)=b^{n}(\zeta)\frac{k^{\alpha-n\mu}(\zeta)}{\sqrt{k^{\alpha-n\mu}(0)}}
	\end{gather*}
	\begin{itemize}\itemsep=0pt
		\item[$(i)$] forms an orthonormal basis in $H^2(\alpha)$ for $n\in \bbZ_{\geq 0}$ and
		\item[$(ii)$] forms an orthonormal basis in $L^2(\alpha)$ for $n\in\mathbb Z$,
	\end{itemize}
	where
	\begin{gather*}
	L^2(\alpha) = \big\{ f \in L^2(\bbT)\colon f \circ \gamma = e^{2\pi i \alpha(\gamma)} f,~\gamma\in\Gamma \big\}.
	\end{gather*}
\end{Theorem}
\begin{proof}
	We only prove $(i)$. $H^2(\alpha)$ can be decomposed into
	\begin{gather}\label{eq:orthDecomJacobi}
		H^2(\alpha)=\{e_0^\alpha\}\oplus H^2_0(\alpha),
	\end{gather}
	where
	$
		H_0^2(\alpha)=\{f\in H^2(\alpha)\colon f(0)=0\}.
	$
	Since
	$
		H_0^2(\alpha)=bH^2(\alpha-\mu),
	$
	 iterating the previous step leads to
	\begin{gather*}
		H^2(\alpha)=\{e_0^\alpha\}\oplus\big\{be_0^{\alpha-\mu}\big\}\oplus\big\{b^2e_0^{\alpha-2\mu}\big\}\oplus\cdots.
	\end{gather*}
	It is easy to see that this system is complete.
\end{proof}

The following theorem describes a one-to-one correspondence between $\Gamma^*$ and $J(E)$.
\begin{Theorem}[\cite{KontaniLast}]	\label{thm:multbyz}
	The multiplication operator by $\fz$ in $L^2(\alpha)$ with respect to the basis $\{e_n^{\alpha}\}$ from Theorem~{\rm \ref{thm:onb}}
	is the following Jacobi matrix $J=J(\alpha)$:
	\begin{gather*}
	\fz e_{n}^{\alpha}=a_{n}(\alpha)e_{n-1}^{\alpha} + b_n(\alpha)e_{n}^{\alpha}+a_{n+1}(\alpha)e^{\alpha}_{n+1},
	\end{gather*}
	where
	\begin{gather*}
	a_n(\alpha)=\mathcal A(\alpha-n\mu), \qquad\mathcal A(\alpha)=(\fz b)(0)\sqrt{\frac{k^{\alpha}(0)}{k^{\alpha+\mu}(0)}}
	\end{gather*}
	and
	\begin{gather*}
	b_n(\alpha)=\mathcal B(\alpha-n\mu),\qquad \mathcal B(\alpha)=\frac{(\fz b)(0)}{b^{\prime}(0)}+
	\left\{ \frac{\left(k^{\alpha}\right)^{\prime}(0)}{k^{\alpha}(0)}- \frac{\left(k^{\alpha+\mu}\right)^{\prime}(0)}{k^{\alpha+\mu}(0)}\right\}
	+\frac{\left(\fz b\right)^{\prime}(0)}{b^{\prime}(0)}.
	\end{gather*}
	This Jacobi matrix $J(\alpha)$ belongs to $J(E)$. Thus, we have a map from $\Gamma^{*}$ to $J(E)$. Moreover, this map is one-to-one.
\end{Theorem}
\begin{Remark}
 Since $S^{-1}J(\alpha)S=J(\alpha-\mu)$, we see that $J(E)$ consists of $p$-periodic Jacobi matrices if and only if $p\mu=\mathbf{0}_{\Gamma^*}$.
\end{Remark}
Note that $p\mu=\mathbf{0}_{\Gamma^*}$ implies that $b^p(\gamma(\zeta))=b^p(\zeta)$. Therefore, the function, $\Phi(z)$, def\/ined by
\begin{gather*}
	 \Phi(z):=\Phi(\fz(\zeta))=b^p(\zeta)
\end{gather*}
is single valued. Since
\begin{itemize}\itemsep=0pt
	\item[(i)] $|\Phi|<1$ in $\overline{\bbC}\setminus E$ and $|\Phi|=1$ on $E$,
	\item[(ii)] $\Phi$ has a zero of multiplicity $p$ at inf\/inity,
\end{itemize}
it is not hard to show that $\Phi$ is the function given by~\eqref{eq:MagicFormulafm}. Moreover, this def\/inition implies that for all $\alpha\in\Gamma^*$
\begin{gather*}
	T_p(J(\alpha))=S^p+S^{-p},
\end{gather*}
which proves one direction of the magic formula.

\subsection{Functional models for periodic GMP matrices}\label{sec:fmGMP}
We have already mentioned in the introduction that the Ahlfors function, $\Psi$, will serve as an analogue of $\Phi$ for general f\/inite gap sets. Due to Lemma \ref{lem:Delta} and its proof, we obtain the following properties of $\Psi$:
\begin{itemize}\itemsep=0pt
	\item[(i)] $|\Psi|<1$ in $\overline{\bbC}\setminus E$ and $|\Psi|=1$ on $E$.
	\item[(ii)] $\Psi(z)=0\gdw z\in\{\bc_1,\dots,\bc_g\}\cup\{\infty\}$.	
\end{itemize}
All this implies that
\begin{gather*}
\log\frac 1{|\Psi(z)|}=G(z)+\sum_{j=1}^g G(z,\bc_j).
\end{gather*}
Therefore, $\Psi(\fz(\zeta))=b(\zeta)\prod\limits_{j=1}^g b_{\bc_j}(\zeta)$. In particular, $\mu+\sum^g_{j=1}\mu_{\bc_j}=\mathbf{0}_{\Gamma^*}$. We def\/ine the permutation $\pi$ by demanding that $\bc_{j}\in(\ba_{\pi(j)},\bb_{\pi(j)})$. Moreover, we f\/ix generators of the group~$\Gamma$, $\{\gamma_j\}_{j=1}^g$, where $\gamma_j$ corresponds to a closed curve, which starts at $\infty$ and passes through the gap $(\mathbf a_{\pi(j)}, \mathbf{b}_{\pi(j)})$; cf.~\cite[Chapter~9, Section~6]{SimonSzego}. Due to the symmetry of $\bbC_+$ and $\bbC_-$ we can choose a fundamental domain, $\cF\subset\bbD$, of the group, which is symmetric w.r.t.\ $\zeta\mapsto\overline \zeta$. Choosing $\zeta_j\in(\cF\cap\fz^{-1}(\bc_j))$, we have $\overline \zeta_j=\gamma_j(\zeta_j)$. Let us f\/ix these $\zeta_{j}\in\bbD$.

The relation \eqref{eq:DeltaAhlfors} together with the factorization of $\Psi$ into Blaschke products suggests to consider the following counterpart of \eqref{eq:orthDecomJacobi}. Let $\beta_n=\alpha-\sum\limits_{k=1}^{n}\mu_{\bc_k}$ and $\eta_n=\sqrt{\exp(-2\pi i\beta_n(\gamma_{n+1}))}$. Then
\begin{gather*}
H^2(\alpha)=\{k^{\alpha}_{\zeta_{1}},\dots, k^{\alpha}_{\zeta_{g}}, k^{\alpha}\}\oplus\Psi H^2(\alpha)
=\{f^{\alpha}_0\}\oplus\dots\oplus\{f^{\alpha}_g\}\oplus\Psi H^2(\alpha),
\end{gather*}
where
\begin{gather*}
f_0^{\alpha}=\eta_0\frac{k^\alpha_{\zeta_{1}}}{\sqrt{k^\alpha_{\zeta_{1}}(\zeta_{1})}}, \qquad
f_1^\alpha=\eta_1\frac{b_{\bc_1} k_{\z_{2}}^{\alpha-\mu_{\bc_1}}}{\sqrt{k_{\z_{2}}^{\alpha-\mu_{\bc_1}}(\z_{2})}},\qquad \dots,
\qquad f_g^\alpha=\frac{\prod_{j=1}^g b_{\bc_j} k^{\alpha+\mu}}{\sqrt{k^{\alpha+\mu}(0)}}.
\end{gather*}
\begin{Remark}
	The unimodular factors $\eta_n$ are chosen such that the matrix representation of multiplication by $\fz$ w.r.t.\ this basis, i.e., the corresponding GMP matrix is real. Note that the square root of $e^{-2\pi i\beta_n(\gamma_{n+1})}$ is def\/ined up to the choice of a multiplicative constant $\pm 1$. Thus, in fact to a~given character $\alpha$ we associate $2^g$ bases and therefore GMP matrices. This is the reason for the identif\/ication $(p_j,q_j)\to (-p_j,-q_j)$ in Theorem~\ref{prop:isospectralGMP}.
\end{Remark}
In order to prove completeness of the system in $L^2(\alpha)$ we show the following lemma.
\begin{Lemma}\label{lem:L2complete}
	Let $\theta\neq 1$ be a character automorphic inner function with character $\chi$. Then
	 \begin{gather*}
		\bigcap_{n=1}^\infty\big\{\theta^{-n}H^2(\alpha+n\chi)\big\}^\perp=\{0\}.
	\end{gather*}
\end{Lemma}
\begin{proof}
	In the following we will use that
	\begin{gather*}
		L^2(\alpha)\ominus H^2(\alpha)=\Theta\overline{H_0^2(\nu-\alpha)},
	\end{gather*}
	where $\Theta$ is the inner part of the derivative of~$b$; cf.~\cite{KontaniLast}. It is character automorphic and we denote its character by~$\nu$. Let $f\in\bigcap_{n=1}^\infty\left\{\theta^{-n}H^2(\alpha+n\chi)\right\}^\perp$. Thus, for all $n\in\bbN$
	\begin{gather*}
		0=\langle f, \theta^{-n} h\rangle=\langle \theta^{n}f, h\rangle, \qquad \forall\, h\in H^2(\alpha+n\chi).
	\end{gather*}
	Therefore, $\theta^nf\in H^2(\alpha+n\chi)^\perp=\Theta\overline{H_0^2(\nu-\alpha-n\chi)}$. Hence, there exists $g^{\nu-\alpha-n\chi}\in H_0^2(\nu-\alpha-n\chi)$ such that $\Theta\overline{f}=\theta^ng^{\nu-\alpha-n\chi}$, i.e.,
	\begin{gather*}
		\Theta\overline{f}\in\bigcap_{n=1}^\infty\theta^n H^2(\nu-\alpha-n\chi).
	\end{gather*}
	For $g\in \bigcap_{n=1}^\infty\theta^n H^2(\nu-\alpha-n\chi)$ and arbitrary $\zeta_0\in\bbD$, we see that
	\begin{gather*}
		|g(\zeta_0)|^2=\big|\langle g,\theta^n\overline{\theta^n(\zeta_0)}k^{\nu-\alpha-n\chi}_{\zeta_0}\rangle\big|^2\leq\|g\|^2|\theta(\zeta_0)|^{2n}k^{\nu-\alpha-n\chi}_{\zeta_0}(\zeta_0).
	\end{gather*}
	Since $k^{\nu-\alpha-n\chi}_{\zeta_0}(\zeta_0)$ is uniformly bounded in $n$ from above, the right-hand side converges to zero as $n\to\infty$ and hence $g(\zeta_0)=0$.
\end{proof}
\begin{Theorem}	\label{thm:onbsmp}
	The system of functions
	\begin{gather*}
	f_{n}^{\alpha}=f_{n}^{\alpha}(\z;\bc_1,\dots,\bc_g)=\Psi^m f_j^\alpha,\qquad
	n=(g+1)m+j, \qquad 0\leq j\leq g,
	\end{gather*}
	\begin{itemize}\itemsep=0pt
		\item[$(i)$] forms an orthonormal basis in $H^2(\alpha)$ for $n\in \bbZ_{\geq 0}$ and
		\item[$(ii)$] forms an orthonormal basis in $L^2(\alpha)$ for $n\in\mathbb Z$.
	\end{itemize}
\end{Theorem}
\begin{proof}
	By construction, this system is orthogonal. The completeness of the second system follows by the previous lemma.
\end{proof}
\begin{proof}[Proof of Theorem~\ref{prop:isospectralGMP}]
	 We show that $A(\alpha)$ is a $\GMP$ matrix. Clearly, $A(\alpha)$ is $g+1$-periodic. Let $P_-^\alpha$ be the projection onto $L^2(\alpha)\ominus H^2(\alpha)$. Then we have
	 \begin{gather*}
	 P_-^\alpha(\fz f_n^\alpha)= p_n(\alpha) f_{-1}^\alpha,\qquad j=0,\dots,g.
	 \end{gather*}
	 Since $\fz$ is self-adjoint and $A(\alpha)$ has constant block-coef\/f\/icients, this shows that $A(\vec p)$ has the right structure.We have
	 \begin{gather*}
	 p_n(\alpha)=\langle \fz f_n^\alpha, f_{-1}^\alpha\rangle=\frac{(\fz b)(0)\eta_n\prod\limits_{k=1}^{n}b_{\bc_k}(0)k^{\beta_n}(0,\zeta_{n+1})}{\sqrt{k^{\beta_n}_{\zeta_{n+1}}(\zeta_{n+1})k^{\alpha+\mu}(0)}}
	 \end{gather*}
	 for $0\leq n\leq g$. Since $k^\alpha(\bar{\zeta_n})=\overline{k^\alpha(\zeta_n,0)}$, we obtain
	 \begin{gather*}
	 	\overline{k^{\beta_n}(0,\zeta_{n+1})}
	 	=k^{\beta_n}(\zeta_{n+1},0)
	 	=\overline{k^{\beta_n}(\gamma_{n+1}(\zeta_{n+1}),0)}
	 	=\eta_n^2k^{\beta_n}(0,\zeta_{n+1}).
	 \end{gather*}
	 Thus, $p_n(\alpha)$ are real. In particular, $p_g(\alpha)>0$. We def\/ine
	 \begin{gather*}	 	q_m(\alpha)=-\frac{\overline{\eta_m}\sqrt{k^{\alpha+\mu}(0)}k^{\beta_m+\mu}(\zeta_{m+1})}{\prod\limits_{k=1}^{m}b_{\bc_k}(0)b(\zeta_{m+1})k^{\beta_m+\mu}(0)\sqrt{k^{\beta_m}_{\zeta_{m+1}}(\zeta_{m+1})}}
	 \end{gather*}
	 Let $n>m$. Note that $\fz f_n^\alpha$ has a simple poles at $\fz^{-1}(\infty)$. Therefore,
	 \begin{gather*}
	 	\fz f_n^\alpha-\frac{(\fz f_n^\alpha b)(0)k^{\alpha+\mu}}{k^{\alpha+\mu}(0)b}\in H^2(\alpha).
	 \end{gather*}
	 Using this and the fact that $(\fz f_n^\alpha)(\zeta_{m+1})=0$, we obtain
	 \begin{gather*}
\begin{split}
& 	\langle\fz f_n^\alpha,f_m^\alpha\rangle= -\frac{\eta_n\overline{\eta_m}}{\sqrt{k^{\beta_{n}}_{\zeta_{n+1}}(\zeta_{n+1})}\sqrt{k^{\beta_{m}}_{\zeta_{m+1}}(\zeta_{m+1})}}\\
& \hphantom{\langle\fz f_n^\alpha,f_m^\alpha\rangle= }{}
	 	\times\frac{(\fz b)(0)\prod\limits_{k=m+1}^{n}b_{\bc_k}(0)k_{\zeta_{n+1}}^{\beta_n}(0)}{k^{\beta_m+\mu}(0)}\frac{k^{\beta_m+\mu}(\zeta_{m+1})}{b(\zeta_{m+1})}
	 	=p_n(\alpha)q_m(\alpha),
\end{split}
	 \end{gather*}
	 by def\/inition. Using in addition that $b(\overline \zeta)=\overline{b( \zeta)}$ we obtain in the same way as before that
	 \begin{gather*}
	 	\frac{\overline{k^{\beta_m+\mu}(\zeta_{m+1},0)}}{\overline{b(\zeta_{m+1})}}=	\frac{\overline{\eta_m}^2k^{\beta_m+\mu}(\zeta_{m+1},0)}{b(\zeta_{m+1})}.
	 \end{gather*}
	 Hence $q_m(\alpha)\in\bbR$. Finally, we consider the diagonal terms, i.e.,
	 \begin{gather*}
	 \langle\fz f_n^\alpha,f_n^\alpha\rangle=\bc_{n+1}-\frac{(\fz b)(0)k_{\zeta_{n+1}}^{\beta_n}(0)}{k_{\zeta_{n+1}}^{\beta_n}(\zeta_{n+1})k^{\beta_n+\mu}(0)}
\frac{k^{\beta_n+\mu}(\zeta_{n+1})}{b(\zeta_{n+1})}=\bc_{n+1}+p_n(\alpha)q_n(\alpha).
	 \end{gather*}
	 The structure of the resolvents given in Def\/inition \ref{def:GMP} follows by the conformal invariance of the Ahlfors function. More specif\/ically, if $w=w_j=\frac{1}{\bc_j-z}$, then $\Psi_j(w):=\Psi(z)$ is the Ahlfors function in the $w$-plane. The given ordering $\bC$ generates the specif\/ic ordering
	 \begin{gather*}
	 \bC_j=\left\{\frac 1{\bc_{j+1}-\bc_j},\dots,\frac 1{\bc_{g}-\bc_j},0,\frac 1{\bc_{1}-\bc_j},\dots,\frac 1{\bc_{j-1}-\bc_j}\right\}
	 \end{gather*}
	 and the multiplication by $w$ is again a periodic GMP matrix (up to an appropriate shift). That is,
	 \begin{gather*}
	 S^{-j}(\bc_j-A(\alpha,\bC))^{-1}S^j\in A(E_j,\bC_j),
	 \end{gather*}
	 where $E_j=\{y=\frac 1{c_j-x}\colon x\in E\}$. This shows that $A(\alpha)\in\GMP(\bC)$. Hence, $A(\alpha)\in A(E)$.
	
Now, we turn to the map $A(E,\bC)\to \Gamma^*$. To $A\in A(E,\bC)$, we associate the resolvent func\-tions~$a_0^2r_+$ and~ $r_-^{-1}$. Due to~\eqref{eq:resolventplus}, Lemma~\ref{lem:rminusSecondSolution} and Theorem~\ref{thm:spectrumOfperiodicA}, $A$ is ref\/lectionless on~$E$. Hence, we can apply the construction of \cite[Sections~3 and~4]{KontaniLast} to obtain $\alpha\in\Gamma^*$. Due to uniqueness of the associated character $\alpha$, this map is one-to-one, up to the identif\/ication $(p_k,q_k)\mapsto(-p_k,-q_k)$.
\end{proof}

\subsection[The magic formula and parametrization of $A(E,\bC)$]{The magic formula and parametrization of $\boldsymbol{A(E,\bC)}$}
Let us turn to the proof of the magic formula and Theorem~\ref{cor:IsospectralAlgebraicManifold}. The following lemma describes the coef\/f\/icients of~$\Delta^A$ in terms of~$A$.
\begin{Lemma}\label{lem:coeffofDeltaA}
	Let $A\in\GMP(\bC)$ be a periodic GMP matrix with generating coefficients $\vbp$ and~$\Delta^A(z)$ its discriminant, cf.~\eqref{def:discriminant}. Then~$\Delta^A$ is a rational function with simple poles at~$\bc_k$ and infinity, i.e.,
	\begin{gather*}
	\Delta^A(z)=d_0+\nu_0z+\sum_{k=1}^{g}\frac{\nu_k}{z-\bc_k}.
	\end{gather*}
	Moreover, the coefficients are given by
	\begin{gather*}
	\nu_0=\frac 1{p_g}, \qquad d_0=-q_g-\nu_0\sum_{j=1}^{g-1} p_jq_j, \qquad \nu_k=\Lambda_k(\vbp)\qquad \text{for} \quad k=1,\dots ,g,
	\end{gather*}
	where $\Lambda_k$ is defined in~\eqref{eq:defLambdaK}.
\end{Lemma}
\begin{proof}
	Considering the residues of $\Delta_A$ at $\bc_k$ and inf\/inity we obtain the coef\/f\/icients $\nu_k$, for $k=0,\dots,g$. Thus it remains to show the formula for $d_0$. To this end, we write
	\begin{gather*}
	\fA(z)=Az+B+\sum_{j=1}^{g}\frac{1}{\bc_j-z}C_j,
	\end{gather*}
	where the constant term is given by
	\begin{gather*}
	B=
	\begin{bmatrix}
	0& -p_g\\
	\frac{1}{p_g}& -q_g
	\end{bmatrix}
	-
	\frac{1}{p_g}
	\sum_{j=1}^{g}
	\begin{bmatrix}
	0& p_{j-1}^2\\
	0& p_{j-1}q_{j-1}
	\end{bmatrix}
	\end{gather*}
	Since $\Delta^ A=\tr\fA$, we also obtain the expression for $d_0$.
\end{proof}
\begin{Remark}
	Note that due to the def\/inition of GMP matrices $\nu_k>0$ for $k=0,\dots,g$. Thus, $\Delta^ A$ maps the upper half plane into itself.
\end{Remark}
\begin{Lemma}\label{lem:deltaAequalDelta}
	Let $E$ be a finite gap set and $\Delta_E$ the corresponding function from Lemma~{\rm \ref{lem:Delta}}. Let $A\in\GMP(\bC)$ be periodic. Then the following are equivalent:
	\begin{itemize}\itemsep=0pt
		\item[$(i)$]	$A\in A(E,\bC)$,
		\item[$(ii)$] $\Delta^A(z)=\Delta_E(z)$,
		\item[$(iii)$] $\Delta^A(A)=\Delta_E(A)$.
	\end{itemize}
\end{Lemma}

\begin{proof}
	Let $\Delta^A(z)=\Delta_E(z)$. By Theorem \ref{thm:spectrumOfperiodicA}, we obtain that
	\begin{gather*}
	\sigma(A)=\big\{z\in\bbC\colon \Delta^A(z)\in[-2,2]\big\}=\{z\in\bbC\colon \Delta_E(z)\in[-2,2]\}=E.
	\end{gather*}
	Thus, $A\in A(E,\bC)$. On the other hand, if $A\in A(E,\bC)$, then $\{z\in\bbC\colon \Delta^A(z)\in[-2,2]\}=E$. By the previous remark and the uniqueness of~$\Delta_E$, we obtain $\Delta^A(z)=\Delta_E(z)$. Hence, $(i)\iff(ii)$. $(i)\implies (ii)$ is clear.
	On the last non-zero diagonal of $\Delta^A(A)$, i.e., $\Delta^A(A)_{j,g+1+j}$ for $j=0,\dots,g$, only one of the summands is non-vanishing. With the notation from the previous lemma, $(iii)$~yields
	\begin{gather*}
		(\nu_k-\lambda_k)(\bc_k-A)^{-1}_{k-1,g+k}=0, \qquad (\lambda_0-\nu_0)A_{g,2g+1}=0.
	\end{gather*}
	Hence, $\Delta^A(z)=\Delta_E(z)$.
\end{proof}
\begin{proof}[Proof of Theorem~\ref{thm:magicFormula}]
	Let $A\in A(E,\bC)$. Due to Proposition~\ref{prop:isospectralGMP}, there exists $\alpha\in\Gamma$ such that $A=A(\alpha)$. Hence, \eqref{eq:DeltaAhlfors} is the magic formula in terms of functional models.
	
	Let $A\in\GMP(\bC)$ satisfy $\Delta_E(A)=S^{g+1}+S^{-(g+1)}$. Na\u{\i}man's lemma (cf.\ \cite[Lemma~8.2.4]{SimonSzego}) implies that~$A$ is periodic. Since by def\/inition $A\in A(\sigma(A),\bC)$, we obtain that
	\begin{gather*}
		\Delta^A(A)=S^{g+1}+S^{-(g+1)}=\Delta_E(A).
	\end{gather*}
	By Lemma \ref{lem:deltaAequalDelta} $A\in A(E,\bC)$.
\end{proof}

\begin{proof}[Proof of Theorem~\ref{cor:IsospectralAlgebraicManifold}]
	Let $A\in A(E,\bC)$. Due to Lemma \ref{lem:deltaAequalDelta}, $\Delta^A(z)=\Delta_E(z)$ and by Lemma~\ref{lem:coeffofDeltaA}, the coef\/f\/icients of $A$ satis\-fy~\eqref{iso101}. On the other hand, if the coef\/f\/icients satis\-fy~$\eqref{iso101}$, then $\Delta^A(z)=\Delta_E(z)$ and hence $A\in\A(E,\bC)$.
\end{proof}

\subsection[The Jacobi f\/low on $A(E,\bC)$]{The Jacobi f\/low on $\boldsymbol{A(E,\bC)}$}
Finally, we would like to explain brief\/ly the idea of the so-called Jacobi f\/low on $\GMP$ matrices, which was another main tool in proving the Killip--Simon theorem for general system of intervals. There is an obvious map $\cF\colon A(E)\to J(E)$ def\/ined by
\begin{gather*}
	\cF A(\alpha)=J(\alpha).
\end{gather*}
The question is, how to f\/ind the coef\/f\/icients of $J(\alpha)$ in terms of coef\/f\/icients of~$A(\alpha)$.
\begin{Proposition}
	To $\alpha\in\Gamma^*$, we associate a $\GMP$ matrix $A(\alpha)$ with coefficients $(\vp(\alpha), \vq(\alpha))$ and a Jacobi matrix $J(\alpha)$ with coefficients $(a_j(\alpha),b_j(\alpha))$. Then we have
	\begin{gather*}
	a_0(\alpha)=\|\vp(\alpha)\|\qquad\text{and}\qquad b_{-1}(\alpha)=p_g(\alpha)q_g(\alpha).
	\end{gather*}
\end{Proposition}
\begin{proof}
	Let $P_+(\alpha)$ be the orthogonal projection onto $H^2(\alpha)$. Since $f_{-1}^\alpha=e_{-1}^\alpha$, we see that
	\begin{gather*}
		a_0(\alpha)=\|P_+(\alpha)\fz e_{-1}^\alpha\|=\|P_+(\alpha)\fz f_{-1}^\alpha\|=\|p(\alpha)\|
	\end{gather*}
	and
	\begin{gather*}
		b_{-1}(\alpha)=\langle e_{-1}^\alpha,\fz e_{-1}^\alpha\rangle=\langle f_{-1}^\alpha,\fz f_{-1}^\alpha\rangle=p_g(\alpha)q_g(\alpha).\tag*{\qed}
	\end{gather*}
\renewcommand{\qed}{}
\end{proof}

Since $S^{-1}J(\alpha)S=J(\alpha-\mu)$, one can f\/ind all coef\/f\/icients of $J(\alpha)$ by f\/inding the coef\/f\/icients of $A(\alpha-\mu)$.
\begin{Definition}
	We def\/ine the Jacobi f\/low on $A(E,\bC)$ as the dynamical system generated by the following map:
	\begin{gather*}
	\mathcal J A(\alpha)=A(\alpha-\mu), \qquad \alpha\in\Gamma^*.
	\end{gather*}
\end{Definition}
For a parametric description of this map see \cite[Theorem~4.5]{YudKillipSimon}.

\subsection*{Acknowledgements}
 The author was supported by the Austrian Science Fund FWF, project no: P25591-N25. He would like to thank his advisor Peter Yuditskii for his guidance and help during the preparation of this paper. Finally, he is grateful to the anonymous referees for their remarks that improved the presentation of the paper.

\pdfbookmark[1]{References}{ref}
\LastPageEnding

\end{document}